\documentclass[reqno,12pt]{article}
\usepackage{amssymb,amsthm,amsmath,amsfonts}
\usepackage{epsfig}

\theoremstyle{plain}
\begingroup

\newtheorem{thm}{Theorem}[section]
\newtheorem{theorem}{Theorem}[section]

\newtheorem{corollary}[thm]{Corollary}
\newtheorem{lemma}[thm]{Lemma}

\endgroup
\theoremstyle{definition}
\newtheorem{defn}{Definition}[section]
\newtheorem{definition}{Definition}[section]

\newtheorem{example}[defn]{Example}

\theoremstyle{remark}

\numberwithin{equation}{section}
\numberwithin{figure}{section}

\DeclareMathOperator{\re}{Re} \DeclareMathOperator{\im}{Im}

\DeclareMathOperator{\Res}{\mathcal{R}}
\DeclareMathOperator*{\res}{\mathrm{Res}}

\def\I{\mathrm{i}}

\def\D{{\mathbb D}}

\def\C{{\mathbb C}}

\begin{document}

\title{The string equation for non-univalent functions
}

\author{
Bj\"orn Gustafsson\textsuperscript{1}
}

\maketitle

\begin{center}
{\it In Memory of Alexander Vasil\'{e}v} 
\end{center}

\begin{abstract}
For conformal maps defined in the unit disk one can define a certain Poisson bracket
that involves the harmonic moments of the image domain.
When this bracket is applied to the conformal map itself together with its conformally reflected map the result is
identically one. This is called the string equation, and it is closely connected to the governing equation, the Polubarinova-Galin equation, for
the evolution of a Hele-Shaw blob of a viscous fluid (or, by another name, Laplacian growth).  In the present paper we investigate
to what extent the string equation makes sense and holds for non-univalent analytic functions.

We give positive answers in two cases:  for polynomials and for a special class of rational functions.

\end{abstract}



 \footnotetext[1]
{Department of Mathematics, KTH, 100 44, Stockholm, Sweden.\\
Email: \tt{gbjorn@kth.se}}


\section{Introduction}

This paper is inspired by fifteen years of collaboration with Alexander Vasil'ev. It gives some details related to a talk given at
the conference ``ICAMI 2017 at San Andr\'es Island, Colombia'', November 26 - December 1, 2017,
partly in honor of Alexander Vasil\'ev.

My collaboration with  Alexander Vasil\'ev  started with some specific questions concerning Hele-Shaw flow and evolved over time into various areas of modern mathematical
physics. The governing equation for the Hele-Shaw flow moving boundary problem we were studying is called the Polubarinova-Galin equation,
after the two Russian mathematicians  P.Ya.~Polubarinova-Kochina and L.A.~ Galin who formulated this equation around 1945. Shortly later,
in 1948, U.P.~Vinogradov and P.P.~Kufarev were able to prove local existence of solutions of the appropriate initial value problem, under the necessary
analyticity conditions. 

Much later, around 2000, another group of Russian  mathematicians, or mathematical physicists, lead by 
M.~Mineev-Weinstein, P.~ Wiegmann, A. ~Zabrodin, considered the Hele-Shaw problem from the point of view
of integrable systems, and the corresponding equation then reappears under the name ``string equation''. See for example
\cite{Wiegmann-Zabrodin-2000}, \cite{Mineev-Zabrodin-2001}, \cite{Kostov-Krichever-Mineev-Wiegmann-Zabrodin-2001}, \cite{Krichever-Marshakov-Zabrodin-2005}.
The integrable system approach appears as a consequence of the discovery 1972 by S.~Richardson \cite{Richardson-1972} that the Hele-Shaw problem has
a complete set of conserved quantities, namely the harmonic moments. See \cite{Vasiliev-2009} for the history of the Hele-Shaw problem in general.
It is not clear whether the name ``string equation'' really refers to
string theory, but it is known that the subject as a whole has connections to for example 2D quantum gravity, and hence is at least
indirectly  related to string theory. In any case, these matters have been a source of inspiration for Alexander Vasil\'ev and myself, and in our first book  \cite{Gustafsson-Vasiliev-2006}   
one of the chapters has the title ``Hele-Shaw evolution and strings''.

The string equation is deceptively simple and beautiful. It reads
\begin{equation}\label{string}
\{f,f^*\}=1,
\end{equation}
in terms of a special Poisson bracket referring to harmonic moments  and with $f$ any normalized conformal map from some reference domain, in our case the unit disk, to the fluid domain
for the Hele-Shaw flow. The main question for this paper now is: if such a beautiful equation as (\ref{string}) holds for all univalent functions, shouldn't it  also hold for non-univalent functions?

The answer, which we shall make precise in some special cases, is that the Poisson bracket does not (always) make sense in the non-univalent case, but that one
can  extend its meaning, actually in several different ways, and after such a step the string equation indeed holds.
Thus the problem is not that the string equation is difficult to prove, the problem is that the meaning of the string equation is ambiguous in the non-univalent case. 
The main results in this paper, Theorems~\ref{thm:polynomialstring} and \ref{thm:QRS}, represent two different meanings of the string equation.

The author wants to thank Irina Markina, Olga Vasilieva, Pavel Gumenyuk, Mauricio Godoy Molina,
Erlend Grong and several others for generous invitations in connection with the
mentioned conference ICAMI 2017, and for warm friendship in general.


\section{The string equation for univalent conformal maps}\label{sec:univalent string}

We consider analytic functions $f(\zeta)$ defined in a neighborhood of the closed unit disk and normalized by $f(0)=0$, $f'(0)>0$.
In addition, we always assume that $f'$ has no zeros on the unit circle. 
It will be convenient to write the Taylor expansion around the origin on the form
$$
f(\zeta)= \sum_{j=0}^{\infty}a_{j}\zeta^{j+1} \quad (a_0>0).
$$
If $f$ is univalent it maps $\D=\{\zeta\in\C: |\zeta|<1\}$ onto a domain $\Omega=f(\D)$. The {\it harmonic moments} for this domain are
$$
M_k = \frac{1}{\pi}\int_\Omega z^k {d}x{d}y, \quad k=0,1,2,\dots.
$$
The integral here can be pulled back to the unit disk and pushed to the boundary there. This gives
\begin{equation}\label{Mk}
M_k =\frac{1}{2\pi\I} \int_{ \D} f(\zeta)^k |f'(\zeta)|^2 d\bar{\zeta} d\zeta
=\frac{1}{2\pi\I} \int_{\partial \D} f(\zeta)^k f^*(\zeta)f'(\zeta) d\zeta,
\end{equation}
where
\begin{equation}\label{involution}
f^*(\zeta)=\overline{f(1/\bar{\zeta})}
\end{equation}
denotes the holomorphic reflection of $f$ in the unit circle.
In the form in (\ref{Mk}) the moments make sense also when $f$ is not univalent.

Computing the last integral in (\ref{Mk}) by residues gives  {Richardson's formula} \cite{Richardson-1972}
for the moments:
\begin{equation}\label{Richardson}
M_k= \sum_{(j_1,\dots,j_k)\geq (0,\dots,0)} (j_0+1) a_{j_0}\cdots a_{j_{k}}\bar{ a}_{j_0+\ldots +j_{k}+k},
\end{equation}
This is a highly nonlinear relationship between the coefficients of $f$ and the moments, and even if $f$ is a polynomial of low degree it is virtually impossible to invert it, to
obtain $a_k=a_k(M_0, M_1, \dots)$, as would be desirable in many situations.  Still there is, quite remarkably,  an explicit expressions for the Jacobi determinant of the change
$(a_0,a_1,\dots)\mapsto (M_0,M_1,\dots)$ when  $f$ restricted to the class of polynomials of a fixed degree. 
This  formula, which was proved by to O.~Kuznetsova and V.~Tkachev \cite{Kuznetsova-Tkachev-2004}, \cite{Tkachev-2005} after an initial conjecture of
C.~Ullemar \cite{Ullemar-1980}, will be discussed in some depth below. 

There are examples of different simply connected domains having the same harmonic moments, see for example \cite{Sakai-1978}, \cite{Sakai-1987},
\cite{Zalcman-1987}.  Restricting to domains having analytic boundary the harmonic
moments are however sensitive for at least small variations of the domain. This can easily be proved by potential theoretic methods.
Indeed, arguing on an intuitive level, an infinitesimal perturbation of the boundary can be represented by a signed measure sitting on the boundary
(this measure representing the  speed of infinitesimal motion). The logarithmic potential of  that measure is a continuous function in the complex plane,
and if the harmonic moments were insensitive for the perturbation then
the exterior part of this potential would vanish. At the same time the interior potential is a harmonic function, and the only way all these conditions can be
satisfied is that the potential vanishes identically, hence also that the measure on the boundary vanishes. 
On a more rigorous level, in the polynomial case the above mentioned Jacobi determinant is indeed nonzero.

The conformal map, with its normalization, is uniquely determined by the image domain $\Omega$ and, as indicated above, the domain is locally encoded in the
sequence the moments $M_0, M_1, M_2,\dots$. Thus  the harmonic moments can be viewed as local coordinates in the space of univalent functions, and we may write
$$
f(\zeta)= f(\zeta; M_0, M_1, M_2,\dots).
$$  
In particular, the derivatives ${\partial f}/{\partial M_k}$ make sense. Now we are in position to define the Poisson bracket. 

\begin{definition}\label{def:poisson}
For any two functions $f(\zeta)=f(\zeta; M_0, M_1, M_2,\dots)$, $g(\zeta)=g(\zeta; M_0, M_1, M_2,\dots)$ which are analytic in a neighborhood of the unit circle 
and are parametrized by the moments we define
\begin{equation}\label{poisson}
\{f,g\}=\zeta \frac{\partial f}{\partial \zeta}\frac{\partial g}{\partial M_0}-\zeta  \frac{\partial g}{\partial \zeta}\frac{\partial f}{\partial M_0}.
\end{equation}
This is again a function analytic in a neighborhood of the unit circle and parametrized by the moments.
\end{definition}

The {\it Schwarz function} \cite{Davis-1974}, \cite{Shapiro-1992} of an analytic curve $\Gamma$ is the unique holomorphic function 
defined in a neighborhood of $\Gamma$ and satisfying
$$
S(z)=\bar{z}, \quad z\in \Gamma.
$$
When $\Gamma=f(\partial\D)$, $f$ analytic in a neighborhood of $\partial\D$, the defining property  of $S(z)$ becomes 
\begin{equation}\label{Sff}
S\circ f = f^*,
\end{equation}
holding then identically in a neighborhood of the unit circle. Notice that $f^*$ and $S$ depend on the moments $M_0, M_1, M_2\dots$, like $f$. The string equation asserts that
\begin{equation}\label{string1}
\{f,f^*\}=1
\end{equation}
in a neighborhood of the unit circle, provided $f$ is univalent in a neighborhood of the closed unit disk.
This result was first formulated and proved in \cite{Wiegmann-Zabrodin-2000} for the case of conformal maps onto an exterior  domain (containing the point of infinity). For conformal maps to bounded 
domains proofs based on somewhat different ideas and involving explicitly the Schwarz function were  given in \cite{Gustafsson-2014}, \cite{Gustafsson-Teodorescu-Vasiliev-2014}. 
For convenience we briefly recall the proof below.

Writing (\ref{Sff}) more explicitly as
$$
f^*(\zeta; M_0, M_1,\dots)=S(f(\zeta; M_0, M_1,\dots); M_0,M_1,\dots)
$$
and using the chain rule when computing $\frac{\partial f^*}{\partial M_0}$ gives, after simplification, 
\begin{equation}\label{fffSf}
\{f,f^*\}=\zeta \frac{\partial f}{\partial \zeta}\cdot ( \frac{\partial S}{\partial M_0}\circ f).
\end{equation}
Next one notices that the harmonic moments are exactly the coefficients in the expansion of a certain  Cauchy integral at infinity:
$$
\frac{1}{2\pi\I} \int_{\partial\Omega}\frac{\bar{w}dw}{z-w} = \sum_{k=0}^\infty \frac{M_k}{z^{k+1}} \qquad (|z|>>1).
$$
Combining this with the fact that the jump of this Cauchy integral across $\partial \Omega$ is $\bar{z}$ it follows that $S(z)$
equals the difference between the analytic continuations of the exterior ($z\in\Omega^e$) and interior  ($z\in \Omega$) functions
defined by the Cauchy integral. Therefore
$$
S(z; M_0,M_1,\dots)= \sum_{k=0}^\infty \frac{M_k}{z^{k+1}} + \text{function holomorphic in }\Omega,
$$
and so, since $M_0, M_1,\dots$ are independent variables,
$$
\frac{\partial S}{\partial M_0}(z;M_0, M_1,\dots )=\frac{1}{ z} + \text{function holomorphic in }\Omega.
$$

Inserting this into (\ref{fffSf}) one finds that $\{f,f^*\}$ is holomorphic in $\D$.  Since the Poisson bracket is invariant under holomorphic reflection
in the unit circle it follows that $\{f,f^*\}$ is holomorphic in the exterior of $\D$  (including the point of infinity) as well, hence it must be constant.
And this constant is  found to be one, proving (\ref{string}).

In the forthcoming sections we wish to allow non-univalent analytic functions in the string equation.  Then the basic ideas in the above proof still work,
but what may happen is that $f$ and $S$ are not determined by the moments $M_0, M_1,\dots$ alone. Since $\partial f/\partial M_0$ is a partial
derivative one has to specify all other independent variables in order to give a meaning to it. So there may be more variables, say 
$$
f(\zeta)=f(\zeta; M_0,M_1,\dots; B_1,B_2,\dots).
$$ 
This does not change the proof very much, but the meaning of the string equation depends on the choice of these extra variables.
Natural choices turn out to be locations of  branch points, i.e., one takes $B_j =f(\omega_j)$, where the $\omega_j\in\D$ denote the zeros of $f'$ inside $\D$. 
One good thing with choosing the branch points as additional variables is that keeping these fixed, as is implicit then in the notation $\partial/\partial M_0$,
means that $f$ in this case can be viewed as a conformal map into a fixed Riemann surface, which will be a branched covering over the complex plane.

There are also other possibilities for giving a meaning to the string equation, like restricting $f$ to the class of polynomials of a fixed degree. 
But then one must allow the branch points to move, so this gives a different meaning to $\partial/\partial M_0$.


\section{Intuition and physical interpretation in the non-univalent case}\label{sec:non-univalent}

As indicated above we shall consider also non-univalent analytic functions as conformal maps,  then into Riemann surfaces above $\C$. 
In general these Riemann surfaces will be  branched covering surfaces, and the non-univalence is then absorbed in the covering projection. 
It is easy to understand that such a Riemann surface, or the corresponding conformal map, will in general 
not be determined by the moments $M_0, M_1, M_2, \dots$ alone.

As a simple example, consider an oriented curve $\Gamma$ in the complex plane
encircling the origin twice (say). In terms of the winding number, or {\it index},
\begin{equation}\label{index}
\nu_\Gamma (z)=\frac{1}{2\pi\I} \oint_{\Gamma} \frac{d\zeta}{\zeta -z} \quad (z\in\C\setminus\Gamma),
\end{equation}
this means that $\nu_\Gamma (0)=2$. Points far away from the origin have index zero, and some other points may have  index one (for example). 
Having only the curve $\Gamma$ available it is natural to define the harmonic moments for the multiply covered (with multiplicities $\nu_\Gamma$) set inside $\Gamma$ as
$$
M_k=\frac{1}{\pi} \int_\C z^k \nu_\Gamma (z) dxdy= \frac{1}{2\pi\I}\int_\Gamma z^k \bar{z}dz, \quad k=0,1,2,\dots.
$$
It is tempting to think of this integer weighted set as a Riemann  surface over (part of) the complex plane.
However, without further information this is not possible. Indeed, since some points have index $\geq 2$ such a covering surface will
have to have branch points, and these have to be specified in order to make the set into a Riemann surface. And only after that it is possible to
speak about a conformal map $f$. Thus $f$ is in general not determined by the moments alone. In the simplest non-univalent cases $f$ will be (locally)
determined by the harmonic moments together with the location of the branch points.

There are actually more problems in the non-univalent case. Even if we specify all branch points, the test class of functions $1, z, z^2,\dots$ used in defining
the moments may be too small since each of these functions take the same value on all sheets above any given point in the complex plane. In order for the
Riemann surface and the conformal map to be determined one would need all analytic functions on the Riemann surface itself as test functions. 

There are several ways out of these problems, and we shall consider, in this paper, two such ways:
\begin{itemize}

\item Restrict $f$ to the class of polynomials of a fixed degree. This turns out to work well, even without specifying the branch points.

\item Restrict $f$ to rational functions which map $\D$ onto {\it quadrature Riemann surfaces} admitting a quadrature identity of a special form. Then the branch points have
to be specified explicitly, but it turns out that the presence of a quadrature identity resolves the problem of the test functions $z^k$ being unable to distinguish between sheets. 

\end{itemize}
Quadrature Riemann surfaces were introduced in \cite{Sakai-1988} and, as will become clear in the forthcoming sections,
they naturally enter the picture. 

The physical interpretation of the string equation is most easily explained with reference to general variations of analytic functions
in the unit disk. Consider an arbitrary smooth variation $f(\zeta)=f(\zeta,t)$, depending on a real parameter $t$. We always keep the normalization
$f(0,t)=0$, $f'(0,t)>0$, and $f$ is assumed to be analytic in a full neighborhood of the closed unit disk, with $f'\ne 0$ on $\partial\D$. 
Then one may define a corresponding Poisson bracket written with a subscript $t$: 
\begin{equation}\label{poissont}
\{f,g\}_t=\zeta \frac{\partial f}{\partial \zeta}\frac{\partial g}{\partial t}-\zeta  \frac{\partial g}{\partial \zeta}\frac{\partial f}{\partial t}.
\end{equation}

This Poisson bracket is itself an analytic function in a neighborhood of $\partial\D$. It is determined by its values on $\partial\D$,
where we have
$$
\{f,f^*\}_t= 2\re[\dot{f} \,\overline{\zeta f'}].
$$
The classical {\it Hele-Shaw flow} moving boundary problem, also called {\it Laplacian growth}, is a particular evolution,  characterized (in the univalent case) by
the harmonic moments being conserved, except for the first one which increases linearly with time, say as
$M_0= 2t+{\rm constant}$. This means that $\dot{f}=2\partial f/\partial M_0$, which makes $\{f,f^*\}_t=2\{f,f^*\}$
and identifies the string equation (\ref{string1}) with the  {\it Polubarinova-Galin equation} 
\begin{equation}\label{PG}
{\rm Re\,} [\dot{f}(\zeta, t)\,\overline{\zeta f'(\zeta,t)} ]=1, \quad \zeta\in\partial\D,
\end{equation}
for the Hele-Shaw  problem. 

Dividing (\ref{PG}) by $|f'|$ gives
\begin{equation*}\label{PG1}
{\rm Re\,} [\dot{f}\cdot\overline{\frac{\zeta f'}{|\zeta f' |}} ]=\frac{1}{|\zeta f'|} \quad \text{on }\partial\D.
\end{equation*}
Here the left member can be interpreted as the inner product between $\dot{f}$ and the unit normal vector on $\partial\Omega=f(\partial \D)$,
and the right member as the gradient of a suitably normalized Green's function of $\Omega=f(\D)$ with pole at the origin. In fact, taking that Green's function
to be $\log|\zeta|$  when pulled back to $\D$ and differentiating this with respect to  $z=f(\zeta)$, to get the gradient, we have
$$
|\nabla G_\Omega|=\big| 2\frac{\partial}{\partial z} \log |\zeta|\big|    =\frac{1}{|\zeta f'(\zeta)|}.
$$

Thus (\ref{PG}) says that $\partial\Omega$ moves in the normal direction with velocity $ |\nabla G_\Omega|$, and for the string equation we then have
$$
2\frac{\partial f}{\partial M_0}\big|_{\rm normal} = \frac{\partial G_\Omega}{\partial n} \quad \text{on }\partial\Omega,
$$  
the subscript ``{\rm normal}'' signifying normal component when considered as a vector on $\partial\Omega$.

The above interpretations remain valid in the non-univalent case, with $G_\Omega$ interpreted as the Green's
function of $\Omega$ regarded as a Riemann surface.  However, as already remarked, the moments $M_k$ do not determine $\Omega$ as a Riemann
surface in this case, also specification of the branch points is needed. 
Thus the string equation represents a whole family of domain evolutions in the
non-univalent case. The most natural of these is the one for which the branch points remain fixed, because this case represents a pure expansion of an
initially given Riemann surface which does not change internally during the evolution.

The general Poisson bracket (\ref{poissont}) enters when differentiating the formula (\ref{Mk}) for the moments $M_k$ with respect to $t$
for a given evolution. For a more general statement in this respect we may replace the function $f(\zeta)^k$ appearing in (\ref{Mk}) by a function $g(\zeta,t)$ which is
analytic in $\zeta$ and depends on $t$ in the same way as $h(f(\zeta,t))$ does, where $h$ is analytic, for example $h(z)=z^k$. 
This  means that $g=g(\zeta,t)$ has to satisfy
\begin{equation}\label{gf}
\frac{\dot{g}(\zeta,t)}{ g'(\zeta,t)}=\frac{\dot{f}(\zeta,t)}{f'(\zeta,t)},
\end{equation}
saying that $g$ ``flows with'' $f$ and locally can be regarded as a time independent function in the image domain of $f$. 

We then have (cf.  Lemma~4.1 in \cite{Gustafsson-Lin-2014})
\begin{lemma}\label{lem:momentvariation1}
Assume that $g(\zeta,t)$ is analytic in $\zeta$ in a neighborhood of the closed unit disk and depends smoothly on $t$ in such a way that (\ref{gf}) holds. Then
\begin{equation}\label{Lie}
\frac{1}{2\pi\I}\frac{d}{d t} \int_\D g(\zeta,t)|f(\zeta,t)|^2 d\bar{\zeta} d\zeta =\frac{1}{2\pi}\int_{0}^{2\pi} g(\zeta,t) \{f,f^*\}_t \,d\theta,
\end{equation}
the last integrand  being evaluated at $\zeta=e^{\I\theta}$. 
\end{lemma}
As a special case, with $g(\zeta,t)=h(f(\zeta,t))$, we have
\begin{corollary}\label{lem:momentvariation}
If $h(z)$ is analytic in a fixed domain containing the closure of $f({\D},t)$ then
\begin{equation*}\label{Lie1}
\frac{1}{2\pi\I}\frac{d}{d t} \int_\D h(f(\zeta,t))|f(\zeta,t)|^2 d\bar{\zeta}d\zeta =\frac{1}{2\pi}\int_{0}^{2\pi} h( f(\zeta,t)) \{f,f^*\}_t \,d\theta.
\end{equation*}
\end{corollary}

\begin{proof}
The proof of (\ref{Lie}) is straight-forward: differentiating under the integral sign and using partial integration we have
$$
\frac{d}{d t} \int_\D g |f'|^2 d\bar{\zeta}d\zeta=\frac{d}{d t} \int_{\partial \D} g f^* f' d\zeta
 =\int_{\partial \D} \left(\dot{g} f^* f' + g\dot{f}^* f' + gf^* \dot{f}'\right)d\zeta
$$
$$
 =\int_{\partial \D} \left(\dot{g} f^* f' + g\dot{f}^* f' -g' f^*\dot{f} -g (f^*)' \dot{f}   \right)d\zeta
$$
$$
 =\int_{\partial \D} \left( (\dot{g}f'- \dot{f}g') f^*+ g( \dot{f}^* f'  - (f^*)' \dot{f}   )\right)d\zeta=\int_{\partial \D} g\cdot \{f,f^*\}_t \,\frac{{d}\zeta}{\zeta},
$$
which is the desired result.

\end{proof}


\section{An example}\label{sec:example}


\subsection{General case}\label{sec:general case}

For constants $a,b,c\in\C$ with $0<|a|<1<|b|$, $c\ne 0$, consider the rational function
\begin{equation}\label{fabc}
f(\zeta)=c\cdot \frac{\zeta(\zeta-a)}{\zeta-b}.
\end{equation}
Here the derivative
$$
f'(\zeta)=c\cdot\frac{\zeta^2-2b\zeta +ab}{(\zeta-b)^2}=c\cdot\frac{(\zeta-\omega_1)(\zeta-\omega_2)}{(\zeta-b)^2}
$$
vanishes for
\begin{equation}\label{omega12}
\omega_{1,2}=b(1\pm \sqrt{1-\frac{a}{b}}), 
\end{equation}
where $\omega_1\omega_2=ab$, $\frac{1}{2}(\omega_1 +\omega_2)=b$. The constant $c$ is to be adapted according to the normalization
$$
f'(0)= \frac{ac}{b}>0.
$$
This fixes the argument of $c$, so the parameters $a$, $b$, $c$ represent $5$ real degrees of freedom for $f$.

We will be interested in choices of $a,b,c$ for which one of the roots $\omega_{1,2}$ is in the unit disk, say $|\omega_1|<1$. Then
$|\omega_2|>1$. The function $f$ is in that case not locally univalent, but can be considered as a conformal map onto a Riemann surface
over $\C$ having a branch point over 
$$
f(\omega_1)={cb}\left(1-\sqrt{1-\frac{a}{b}}\right)^2 =\frac{c \,\omega_1^2}{b},
$$
where we, as a matter of notation, let $\omega_1$ correspond to the minus sign in (\ref{omega12}) (this is natural in the case $0<a<1<b$).
The holomorphically reflected function is
$$
f^*(\zeta)=\bar{c} \cdot\frac{1-\bar{a}\zeta}{\zeta(1-\bar{b}\zeta)}.
$$

Let $h(z)$ be any analytic (test) function defined in a neighborhood of the closure of $f(\D)$, for example $h(z)=z^k$, $k\geq 0$. Then,
denoting by $\nu_f$ the index of $f(\partial\D)$, see (\ref{index}), we have
$$
\frac{1}{\pi} \int_\C h \nu_f {d}x{d}y = \frac{1}{2\pi\I} \int_\D h(f(\zeta)) |f'(\zeta)|^2 d\bar{\zeta} d\zeta
= \frac{1}{2\pi\I} \int_{\partial \D}h(f(\zeta)) f^*(\zeta) f'(\zeta)  d\zeta
$$
$$
=\res_{\zeta=0} h(f(\zeta)) f^*(\zeta) f'(\zeta)  d\zeta + \res_{\zeta=1/\bar{b}} h(f(\zeta)) f^*(\zeta) f'(\zeta)  d\zeta
$$
$$
= |c|^2\left( \frac{a}{b} \,h(f(0)) + \frac{(\bar{a}-\bar{b})(1-2|b|^2 +a\bar{b}|b|^2)}{\bar{b}(1-|b|^2)^2}\, h(f(1/\bar{b}))\right).
$$
In summary,
\begin{equation}\label{qihf}
\frac{1}{\pi} \int_\C h \nu_f {d}x{d}y = A h(f(0))+ Bh(f(1/\bar{b})),
\end{equation}
where
\begin{equation}\label{AB}
A=|c|^2 \frac{a}{b},\quad 
B= |c|^2\frac{(\bar{a}-\bar{b})(1-2|b|^2 +a\bar{b}|b|^2)}{\bar{b}(1-|b|^2)^2}.
\end{equation}

For the harmonic moments (with respect to the weight $\nu_f$)  this gives
\begin{align*}\label{moments}
M_0 =& A+B,\\
M_k= & B f(1/\bar{b})^k, \quad k=1,2,\dots.
\end{align*}
Here only $M_0, M_1, M_2$ are needed since the $M_k$ lie in geometric progression from $k=1$ on. 
From these three moments, $A$, $B$ and $f(1/\bar{b})$ can be determined
provided $M_1\ne 0$, and after that $a,b,c$ can be found, at least generically. Thus the moments $M_0, M_1,M_2$ actually suffice to locally determine $f$, 
and we can write
$$
f(\zeta)=f(\zeta; M_0, M_1, M_2),
$$
provided $f$ is known  {\it a priori} to be of the form (\ref{fabc}) with $M_1\ne 0$. However, as will be seen below, 
when we specialize to the case $M_1=0$ things change.

The quadrature Riemann surface picture enters when one starts from the second member in the computation above and considers $g(\zeta)=h(f(\zeta))$
as an independent test function on the Riemann surface, thus allowing $g(\zeta_1)\ne g(\zeta_2)$ even when $f(\zeta_1)=f(\zeta_2)$. The quadrature identity
becomes
\begin{equation}\label{qig}
 \frac{1}{2\pi\I} \int_\D g(\zeta) |f'(\zeta)|^2 d\bar{\zeta} d\zeta
=A g(0)+ Bg(1/\bar{b}),
\end{equation}
for $g$ analytic and integrable (with respect to the weight $|f'|^2$) in the unit disk.


\subsection{First subcase}\label{sec:first case}

There are two cases in this example which are of particular interest. These represent instances of $M_1=0$. The first case is when $a=1/\bar{b}$. 
This does not change (\ref{qig}) very much, it is only that the two weights become equal: 
\begin{equation*}
 \frac{1}{2\pi\I} \int_\D g(\zeta) |f'(\zeta)|^2 d\bar{\zeta} d\zeta
=A g(0)+ Ag(1/\bar{b}),
\end{equation*}
where $A=B={|c|^2}/{|b|^2}$ and the normalization for $c$ becomes $c>0$.
However,  (\ref{qihf}) changes more drastically because the two quadrature nodes now
lie over the same point in the $z$-plane. Indeed $f(1/\bar{b})=0=f(0)$, so (\ref{qihf}) effectively becomes a one point identity:
$$
\frac{1}{\pi} \int_\C h \nu_f {d}x{d}y = 2A h(0)
$$
and, for the moments,
\begin{equation}\label{M2A}
M_0=2A, \quad  M_1=M_2=\dots =0.
\end{equation}

Clearly knowledge of these are not enough to determine $f$. This function originally had $5$ real degrees of freedom. Two of them were 
used in condition $a=1/\bar{b}$, but there still remain three, and $M_0=2A$ is only one real equation. So something more would be needed,
for example knowledge of the location of the branch point $B_1=f(\omega_1)$. We have
$$
\omega_1=b\left( 1-\sqrt{1-\frac{1}{|b|^2}}\,\right), \quad c=|b|\sqrt{\frac{M_0}{2}},
$$
by which 
\begin{equation}\label{branchpoint}
B_1=b|b|\sqrt{\frac{M_0}{2}}\left(1- \sqrt{1-\frac{1}{|b|^2}}\,\right)^2.
\end{equation}
This equation can be solved for $b$ in terms of $B_1$ and $M_0$. Indeed,  by some elementary calculations one finds that
$$
b=\frac{B_1}{2|B_1|}\left( (\frac{2|B_0|^2}{M_0})^{1/4} + \frac{2|B_0|^2}{M_0})^{-1/4} \right),
$$
and after substitution of $a=1/\bar{b}$, $b$ and $c$ one then has  $f$ explicitly on the form
$$
f(\zeta)=f(\zeta; M_0; B_1).
$$

By (\ref{M2A}) the moment sequence is the same as that for the  disk $\D(0,\sqrt{2A})$. One way to understand that is to observe that
$$
f(\zeta) = -\frac{c}{\bar{b}} \cdot \zeta\cdot \frac{1-\bar{b}\zeta}{\zeta-b}
$$
is a function which maps $\D$ onto the disk $\D(0, \sqrt{A})$ covered twice. In other words,  $\nu_f=2\chi_{\D(0,\sqrt{A})}$. 
Note that the disk $\D(0, \sqrt{A})=\D(0, \sqrt{M_0/2})$ depends only on $M_0$, not on $B_1$, so varying  just $B_1$ keeps $f(\partial\D)$ fixed as a set.


\subsection{Second subcase}\label{sec:second case}

The second interesting case is when $\omega_1 =1/\bar{b}$. This means that the quadrature node $1/\bar{b}$ is at the same time a branch point.
What also happens is that the quadrature node looses its weight: one gets $B=0$. The quadrature node is still there, but it is  only ``virtual''. (In principle it can
be restored by allowing meromorphic test functions with a pole at the point, as the weight $|f'|^2$ certainly allows, but we shall not take such steps.)
Thus we have again a one node quadrature identity, but this time in a more true sense,
namely that it is such a quadrature identity on the Riemann surface itself:
\begin{equation}\label{divisor}
\frac{1}{2\pi\I} \int_\D g(\zeta)|f(\zeta)|^2 d\bar{\zeta}d\zeta = A g(0),
\end{equation}
where 
$$
A=2\frac{|c|^2}{|b|^2}-\frac{|c|^2}{|b|^4}=|c|^2|\omega_1|^2(2-|\omega_1|^2).
$$
Of course we also have
$$
\frac{1}{\pi} \int_\C h \nu_f {\rm d}x{\rm d}y = Ah(0).
$$
and
\begin{equation}\label{M0A}
M_0=A, \quad  M_1=M_2=\dots =0.
\end{equation}

Here again the moments do not suffice to identify $f$. Indeed, we have now a one parameter family of functions $f$ satisfying (\ref{M0A}) with the same value
of $A$. Explicitly this becomes, in terms of $\omega_1=1/\bar{b}$, which we keep as the free parameter,
$$
f(\zeta)=C\cdot\frac{\zeta(2|\omega_1|^2- |\omega_1|^4-\bar{\omega}_1\zeta)}{1-\bar{\omega}_1 \zeta},
$$
where
$$
C =\frac{\sqrt{M_0}}{|\omega_1|\sqrt{2-|\omega_1|^2}}.
$$

The branch point is
$$
B_1=f(\omega_1)
=\frac{\omega_1|\omega_1|\sqrt{M_0}}{\sqrt{2-|\omega_1|^2}}.
$$
Since $|\omega_1|<1$ we have $|B_1|<\sqrt{M_0}$.
The above relationship can be inverted to give $\omega_1$ in terms of $B_1$ and $M_0$:
$$
\omega_1= \frac{B_1}{|B_1|}\sqrt{-\frac{|B_1|^2}{2M_0}+\sqrt{\frac{|B_1|^4}{4M_0^2}+\frac{2|B_1|^2}{M_0}}}.
$$
Thus one can explicitly  write $f$ on the form $f(\zeta)=f(\zeta;M_0;B_1)$ also in the present case. 

It is interesting to also compute $f'(\zeta)$. One gets 
$$
f'(\zeta)=C\cdot \frac{(\zeta-\omega_1)(\zeta-\frac{2}{\bar{\omega}_1}+\omega_1)}{(\zeta-\frac{1}{\bar{\omega}_1})^2},
$$
which is, up to a constant factor, the  contractive zero divisor in Bergman space corresponding to the
zero $\omega_1\in \D$, alternatively, the reproducing kernel for those $L^2$-integrable analytic
functions in $\D$ which vanish at $\omega_1$. See  \cite{Hedenmalm-1991}, \cite{Hedenmalm-Korenblum-Zhu-2000}
for these concepts in general. Part of the meaning in the present case is simply that (\ref{divisor}) holds.


\section{The string equation for polynomials}\label{sec:polynomials}

We now restrict to polynomials, of a fixed degree $n+1$:
\begin{equation}\label{fpolynomial}
f(\zeta)= \sum_{j=0}^{n}a_{j}\zeta^{j+1}, \quad a_0>0.
\end{equation}
The derivative is of degree $n$, and we denote its coefficients by $b_j$:
\begin{equation}\label{fderivative}
f'(\zeta)=\sum_{j=0}^n b_j\zeta^j =\sum_{j=0}^n (j+1)a_j \zeta^j, 
\end{equation}
It is obvious from Definition~\ref{def:poisson}
that whenever the Poisson bracket (\ref{poisson}) makes sense (i.e., whenever $\partial f/\partial M_0$ makes
sense), it will vanish if $f'$ has zeros at two points which are reflections
of each other with respect to the unit circle. Thus the string equation cannot hold in such cases. The main result
in this section says that for polynomial maps this is the only exception: the string equation makes sense and
holds whenever $f'$ and $f'^*$ have no common zeros.

Two polynomials having common zeros is something which can be tested by the classical resultant, which vanishes
exactly in this case. Now $f'^*$ is not really a polynomial, only a rational function, but one may work with the polynomial $\zeta^n f'^*(\zeta)$
instead. Alternatively, one may use the {\it  meromorphic resultant}, which applies to meromorphic functions on a compact Riemann surface, 
in particular rational functions.  Very briefly expressed, the meromorphic resultant $\Res (g,h)$
between two meromorphic functions $g$ and $h$ is defined as  the multiplicative action of one of the functions on the divisor of the other. 
The second member of (\ref{RR}) below gives an example of the multiplicative action of $h$ on the divisor of $g$. See \cite{Gustafsson-Tkachev-2009} for further details.

We shall need the meromorphic resultant only in the case of two rational functions of the form $g(\zeta)=\sum_{j=0}^n b_j\zeta^j $ and $h(\zeta)=\sum_{k=0}^n c_k \zeta^{-k}$,
and in this case it is closely related to the ordinary polynomial resultant $\Res_{\rm pol}$ (see \cite{Waerden-1940})
for the two polynomials $g(\zeta)$ and $\zeta^n h(\zeta)$. Indeed,
denoting by $\omega_1,\dots, \omega_n$ the zeros of $g$, the divisor of $g$ is the formal sum $(\omega_1)+\dots +(\omega_n)-n(\infty)$,
noting that $g$ has a pole of order $n$ at infinity. This gives  the meromorphic resultant, and its relation to the polynomial resultant, as 
\begin{equation}\label{RR}
\Res (g,h) =\frac{h(\omega_1)\cdot \dots \cdot h(\omega_n)}{h(\infty)^n}
=\frac{1}{b_0^n c_0^n} \Res_{\rm pol} (g(\zeta), \zeta^n h(\zeta)).
\end{equation}

The main result, to be described below, is an interplay between the Poisson bracket, the resultant and
the Jacobi determinant between the moments and the coefficients of $f$ in (\ref{fpolynomial}).
The theorem is mainly due to O.~Kuznetsova and V.~Tkachev \cite{Kuznetsova-Tkachev-2004}, \cite{Tkachev-2005},
only the statement about the string equation is (possibly) new.
One may argue that this string equation can actually be obtained from the string equation for univalent polynomials
by ``analytic continuation'', but we think that writing down an explicit proof in the non-univalent case really clarifies the nature of the string equation.
In particular the proof shows that the string equation is not an entirely trivial identity.

\begin{theorem}\label{thm:polynomialstring}
With $f$ a polynomial as in (\ref{fpolynomial}), the identity 
\begin{equation}\label{jacobi}
\frac{\partial(\bar{M}_{n}, \dots \bar{M}_1, M_0 , M_1,
\dots, M_n)}{\partial(\bar{a}_{n},\dots, \bar{a}_1, a_0,
a_1, \dots, a_{n} )}
=2a_0^{n^2+3n+1}\Res(f',f'^*)
\end{equation}
holds generally. It follows that the derivative $\partial f/\partial M_0$ makes sense whenever  $\Res(f',f'^*)\ne 0$, and then also  the string equation
\begin{equation}\label{stringthm}
\{f,f^*\}=1
\end{equation} 
holds.
\end{theorem}

\begin{proof}
For the first statement we essentially follow the proof given in \cite{Gustafsson-Tkachev-2009b}, but add some
details which will be necessary for the second statement.

Using Corollary~\ref{lem:momentvariation} we shall first investigate how the moments change under
a general variation of $f$, i.e., we let $f(\zeta)=f(\zeta,t)$ depend smoothly on a real parameter $t$. Thus
$a_j=a_j(t)$, $M_k=M_k(t)$, and derivatives with respect to $t$ will often be denoted by a dot.
For the Laurent series of any function $h(\zeta)=\sum_i c_i \zeta^i$ we denote by ${\rm coeff}_i (h)$ the coefficient 
of $\zeta^i$:  
$$
{\rm coeff}_i (h)=c_i=\frac{1}{2\pi\I} \oint_{|\zeta|=1} \frac{h(\zeta)d\zeta}{\zeta^{i+1}}.
$$
By Corollary~\ref{lem:momentvariation} we then have, for $k\geq 0$,
$$
\frac{d}{dt}{M}_k=\frac{1}{2\pi\I}\frac{d}{d t} \int_\D f(\zeta,t)^k |f(\zeta,t)|^2d\bar{\zeta}d\zeta =\frac{1}{2\pi}\int_{0}^{2\pi} f^k \{f,f^*\}_t \,d\theta
$$
$$
={\rm coeff}_0 (f^k \{f,f^*\}_t)= \sum_{i= 0}^n {\rm coeff}_{+i} (f^k) \cdot {\rm coeff}_{-i} (\{f,f^*\}_t).
$$
Note that $f(\zeta)^k$ contains only positive powers of $\zeta$ and that $\{f,f^*\}_t$ contains powers with exponents in the interval $-n\leq i\leq n$ only.

In view of (\ref{fpolynomial})  the matrix
\begin{equation}\label{vkifk}
v_{ki}=  {\rm coeff}_{+i} (f^k) \quad (0\leq k,i\leq n)
\end{equation}
is upper triangular, i.e., $v_{ki}=0$ for $0\leq i<k$, with  diagonal elements being powers of $a_0$:
$$
v_{kk}= a_0^k.
$$

Next we shall find the coefficients of the Poisson bracket. These will involve the coefficients  $b_k$ and $\dot{a}_j$, but also their complex conjugates. 
For a streamlined treatment it is convenient to introduce coefficients with negative indices to represent  the complex conjugated quantities. The same for the moments.
Thus we define, for the purpose of this proof and the forthcoming Example~\ref{ex:matrix},
\begin{equation}\label{Mab}
M_{-k}=\bar{M}_k, \quad a_{-k}=\bar{a}_k, \quad b_{-k}=\bar{b}_k \quad (k>0).
\end{equation}
Turning points are the real quantities  $M_0$ and $a_0=b_0$. 

In this notation the expansion of the Poisson bracket becomes
\begin{equation}\label{ffff}
\{f,f^*\}_t
= f'(\zeta)\cdot\zeta \dot{f}^*(\zeta) +  f'^*(\zeta)\cdot \zeta^{-1}\dot{f}(\zeta)
\end{equation}
$$
=\sum_{\ell,j\geq 0}   b_\ell \dot{{\bar{a}}}_{j}\zeta^{\ell-j}+\sum_{\ell,j\leq 0}{\bar{b}}_{\ell} \dot{a}_j\zeta^{j-\ell}
=\sum_{\ell\geq 0,\, j\leq 0}   b_\ell \dot{{a}}_{j}\zeta^{\ell+j}+\sum_{\ell\leq 0, \, j\geq 0}{b}_{\ell} \dot{a}_j\zeta^{\ell+j}
$$
$$
=b_0\dot{a}_0 +\sum_{\ell\cdot j\leq 0}  b_\ell \dot{{a}}_{j}\zeta^{\ell+j} =b_0\dot{a}_0 +\sum_i\left( \sum_{\ell\cdot j\leq 0,\, \ell +j=-i}  b_\ell \dot{{a}}_{j}\right)\zeta^{-i}.
$$
The last summation runs over  pairs of indices $(\ell,j)$ having opposite sign (or at least one of them being zero) and adding up to $-i$. 
We presently need only to consider the case $i\geq 0$.
Eliminating $\ell$ and letting $j$ run over those values for which $\ell\cdot j\leq 0$ we therefore get
\begin{equation*}\label{coeff}
{\rm coeff}_{-i} (\{f,f^*\}_t)= b_{0}\, \dot{a}_0\,\delta_{i0}+\sum_{j\leq -i} b_{-(i+j)}\, \dot{a}_j +\sum_{j\geq 0} b_{-(i+j)}\, \dot{a}_j .
\end{equation*}
Here $\delta_{ij}$ denotes the Kronecker delta. Setting,  for $i\geq 0$,
$$
u_{ij} =
\begin{cases} 
b_{-(i+j)} + b_{0} \delta_{i0} \delta_{0j}, \quad & \text{if } -n\leq j \leq -i \text{ or } 0\leq j\leq n, \\
0 & \text{in remaining cases}
\end{cases}
$$
we thus have 
\begin{equation}\label{coeff}
{\rm coeff}_{-i} (\{f,f^*\}_t)=\sum_{j=0}^n u_{ij} \dot{a}_j.
\end{equation}

Turning to the complex conjugated moments we have, with $k< 0$, 
$$
\dot{M}_{k}= \dot{\bar{M}}_{-k}= \sum_{i\leq 0} \overline{{\rm coeff}_{-i} ({f}^{-k})} \cdot \overline{ {\rm coeff}_{+i} (\{f,f^*\}_t)}.
$$
Set, for $k<0$, $i\leq 0$, 
$$
v_{ki}= \overline{ {\rm coeff}_{-i} (f^{-k})}.
$$
Then $v_{ki}=0$ when $k<i\leq 0$, and $v_{kk}=a_0^{-k}$.
To achieve the counterpart of (\ref{coeff}) we define, for $i\leq 0$,
$$
u_{ij} =
\begin{cases} 
b_{-(i+j)} + b_{0} \delta_{i0} \delta_{0j}, \quad & \text{if } -n\leq j \leq 0\text{ or } -i \leq j\leq n, \\
0 & \text{in remaining cases}.
\end{cases}
$$
This gives, with $i\leq 0$,
$$
\overline{{\rm coeff}_{+i} (\{f,f^*\}_t)}=\sum_{j=0}^n u_{ij} \dot{a}_j.
$$

As a summary we have, from  (\ref{vkifk}), (\ref{coeff}) and from  corresponded conjugated equations,
\begin{equation}\label{Mkaj}
\dot{M}_k=\sum_{-n\leq i,j \leq n}v_{ki} u_{ij} \dot{a}_j, \quad -n\leq k\leq n,
\end{equation} 
where
\begin{align*}
v_{ki}&= {\rm coeff}_{+i}(f^k) \quad && \text{when } 0\leq k\leq i,\\
v_{ki}&= \overline{{\rm coeff}_{-i}({f}^{-k})} \quad&&\text{when } i\leq k<0,\\
v_{ki}&=0 \quad &&\text{in remaining cases},\\
u_{ij} &= b_{-(i+j)} + b_{0} \delta_{i0}\delta_{0j}\quad && \text{in index intervals made explicit above},\\
u_{ij}&=0 \quad &&\text{in remaining cases}.
\end{align*}

We see that the full matrix $V=(v_{ki})$ is triangular in each of the two blocks along the main diagonal
and vanishes completely in the two remaining blocks. Therefore, its determinant is simply the product of the
diagonal elements. More precisely this becomes
\begin{equation}\label{Va}
\det V = a_0^{n(n+1)}.
\end{equation}

The matrix $U=(u_{ij})$ represents the linear dependence of the bracket $\{f,f^*\}_t$ on $f'$ and $f'^*$, and it acts 
on the column vector with components $\dot{a}_j$, then representing the linear dependence on  $\dot{f}$ and $\dot{f}^*$.
The computation started at (\ref{ffff})  can thus be finalized as
\begin{equation}\label{ffua}
\{f,f^*\}_t= \sum_{-n\leq i,j \leq n}u_{ij} \dot{a}_j \zeta^{-i}.
\end{equation}

Returning to (\ref{Mkaj}), this equation says that the matrix of partial derivatives $\partial M_k/ \partial a_j$ equals the matrix product $VU$, in
particular that
$$
\frac{\partial({M}_{-n}, \dots {M}_{-1}, M_0 , M_1,
\dots, M_n)}{\partial({a}_{-n},\dots, {a}_{-1}, a_0,
a_1, \dots, a_{n} )}
=\det V\cdot \det U.
$$
The first determinant was already computed above, see (\ref{Va}). It remains to connect $\det U$ to the meromorphic resultant $\Res(f',f'^*)$. 

For any kind of evolution, $\{f,f^*\}_t$ vanishes whenever $f'$ and $f'^*$ have a common zero. 
The meromorphic resultant $\Res (f', f'^*)$ is a complex number which has the same non-vanishing
properties as $\{f,f^*\}_t$, and it is in a certain sense minimal with this property. 
From this one may expect that the determinant of $U$ is simply a multiple of the resultant. Taking homogenieties  into account the constant of proportionality
should be $b_0^{2n+1}$, times possibly some numerical factor. The precise formula in fact turns out to be
\begin{equation}\label{URes}
\det U = 2 b_0^{2n+1} \Res (f', f'^* ).
\end{equation}

One way to prove it is to connect $U$ to the  Sylvester matrix $S$ associated to the polynomial resultant $\Res_{\rm pol}(f'(\zeta), \zeta^n f'^*(\zeta))$. 
This matrix is of size $2n\times 2n$. By some operations with rows and columns (the details are given in \cite{Gustafsson-Tkachev-2009b}, and will in addition
be illustrated in the example below) one finds that
$$
\det U = 2b_0 \det S.
$$
From this (\ref{URes}) follows, using also (\ref{RR}).

Now, the string equation is an assertion about a special evolution.
The string equation says that $\{f,f^*\}_t=1$ for that kind of evolution for which $\partial/\partial t $ means $\partial/ \partial M_0$, in other words
in the case that $\dot{M}_0=1$ and $\dot{M}_k=0$ for $k\ne 0$.  By what has already been proved,
a unique such evolution exists with $f$ kept on the form (\ref{fpolynomial}) as long as $\Res (f',f'^*)\ne 0$.

Inserting $\dot{M}_k=\delta_{k0}$ in (\ref{Mkaj}) gives
\begin{equation}\label{uvadelta}
\sum_{-n\leq i,j \leq n}v_{ki} u_{ij} \dot{a}_j=\delta_{k0}, \quad -n\leq k\leq n.
\end{equation}
It is easy to see from the structure of the matrix $V=(v_{ki})$ that the $0$:th column of the inverse matrix $V^{-1}$, which is sorted out when $V^{-1}$
is applied to the right member in (\ref{uvadelta}), is simply 
the unit vector with components $\delta_{k0}$. Therefore (\ref{uvadelta}) is equivalent to
\begin{equation}\label{vadelta}
\sum_{-n\leq j \leq n} u_{ij} \dot{a}_j=\delta_{i0}, \quad -n\leq i\leq n.
\end{equation}
Inserting this into (\ref{ffua}) shows that the string equation indeed holds.
\end{proof}

\begin{example}\label{ex:matrix}
To illustrate the above proof, and the general theory, we compute everything explicitly when $n=2$, i.e., with
$$
f(\zeta)=a_0 \zeta +a_1 \zeta^2 +a_2 \zeta^3.
$$
We shall keep the convention (\ref{Mab}) in this example. Thus
\begin{align*}
f'(\zeta)&=b_0 +b_1 \zeta +b_2 \zeta^2 = a_0+ 2a_1 \zeta +3a_2 \zeta^2,\\
f^*(\zeta)&=a_0\zeta^{-1} + {a}_{-1} \zeta^{-2} + {a}_{-2} \zeta^{-3},
\end{align*}
for example. When the equation (\ref{Mkaj}) is written as a matrix equation it becomes (with zeros represented by blanks)
\begin{equation}\label{Mdot}
\begin{pmatrix} \dot{{M}}_{-2}\\
\dot{{M}}_{-1}\\
\dot{M}_{0}\\
\dot{M}_{1}\\
\dot{M}_{2}\\
\end{pmatrix}
=
\begin{pmatrix}
a_0^2 &  &  & &   \\
 {a}_{-1}& a_0 &  &  &   \\
 &  & 1 &  & \\
 &  &  & a_{0} & {a}_{1}   \\
 &  &  &  & a_0^2  \\
\end{pmatrix}
\begin{pmatrix}
 &  & b_2 & & b_0  \\
 & b_2 & b_1 & b_0 & {b}_{-1}   \\
b_{2} & b_{1}  & 2b_0 & {b}_{-1} &  {b}_{-2}  \\
b_1& b_0&{b}_{-1} & {b}_{-2} &  \\
b_0& &{b}_{-2} &  & \\
\end{pmatrix}
\begin{pmatrix} \dot{{a}}_{-2}\\
\dot{{a}}_{-1}\\
\dot{a}_{0}\\
\dot{a}_{1}\\
\dot{a}_{2}\\
\end{pmatrix}
\end{equation}
Denoting the two $5\times 5$ matrices by $V$ and $U$ respectively it follows that the corresponding Jacobi determinant is 
$$
\frac{\partial({M}_{-2}, {M}_{-1}, M_0 , M_1, M_2)}{\partial({a}_{-2},{a}_{-1}, a_0,
a_1, a_{2} )}
=\det V\cdot \det U=a_0^6 \cdot\det U. 
$$

Here $U$ can essentially be identified with the Sylvester matrix for  the resultant $ \Res(f' f'^*)$.
To be precise, 
\begin{equation}\label{US}
\det U = 2b_0 \det S,
\end{equation}
where $S$ is the classical Sylvester matrix associated to the two polynomials $f'(\zeta)$ and $\zeta^2 f'^*(\zeta)$, namely
$$
S=
\begin{pmatrix}
 & b_2 &   & b_0  \\
 b_2 & b_1 & b_0 & {b}_{-1}   \\
 b_{1}  & b_0 & {b}_{-1} &  {b}_{-2}  \\
 b_0&  & {b}_{-2} &  \\
\end{pmatrix}.
$$

As promised in the proof above, we shall explain in this example the column operations on $U$  leading from $U$ to $S$,
and thereby proving (\ref{US}) in the case $n=2$ (the general case is similar). The matrix $U$ appears in (\ref{Mdot}).
Let $U_{-2}$, $U_{-1}$, $U_{0}$, $U_1$, $U_2$ denote the columns of $U$.
We make the following change of $U_0$:
$$
U_0 \mapsto \frac{1}{2}U_0-\frac{1}{2b_0} (b_{-2} U_{-2}+ b_{-1}U_{-1} -b_1 U_1 -b_2 U_2)
$$
The first term makes the determinant become half as big as it was before, and the other terms do not affect the 
determinant at all. The new matrix is the $5\times 5$ matrix
$$
\begin{pmatrix}
 &  & b_2 & & b_0  \\
 & b_2 & b_1 & b_0 & {b}_{-1}   \\
b_{2} & b_{1}  & b_0 & {b}_{-1} &  {b}_{-2}  \\
b_1& b_0& & {b}_{-2} &  \\
b_0& & &  & \\
\end{pmatrix}
$$
which has $b_0$ in the lower left corner, with the complementary
$4\times 4$ block being  exactly $S$ above. From this (\ref{US}) follows.

The string equation (\ref{stringthm}) becomes, in terms of coefficients
and with $\dot{a}_j$ interpreted as $\partial a_j/\partial M_0$, 
the  linear equation
$$
\begin{pmatrix}
a_0^2 &  &  & &   \\
 {a}_{-1}& a_0 &  &  &   \\
 &  & 1 &  & \\
 &  &  & a_{0} & {a}_{1}   \\
 &  &  &  & a_0^2  \\
\end{pmatrix}
\begin{pmatrix}
 &  & b_2 & & b_0  \\
 & b_2 & b_1 & b_0 & {b}_{-1}   \\
b_{2} & b_{1}  & 2b_0 & {b}_{-1} &  {b}_{-2}  \\
b_1& b_0&{b}_{-1} & {b}_{-2} &  \\
b_0& &{b}_{-2} &  & \\
\end{pmatrix}
\begin{pmatrix}
\dot{{a}}_{-2}\\
\dot{{a}}_{-1}\\
\dot{a}_{0}\\
\dot{a}_{1}\\
\dot{a}_{2}\\
\end{pmatrix}
=
\begin{pmatrix} 
0\\
0\\
1\\
0\\
0\\
\end{pmatrix}
$$
Indeed, in view of (\ref{Mdot}) this equation characterizes the $\dot{a}_i$ as those
belonging to an evolution such that $\dot{M}_0=1$, $\dot{M}_k=0$ for $k\ne 0$.
As remarked in the step from (\ref{uvadelta}) to (\ref{vadelta}), the first matrix, $V$,
can actually be removed in this equation.

\end{example}


\section{The string equation on quadrature Riemann surfaces}\label{sec:quadratureRS}

Next we shall generalize the above result for polynomials to certain kinds of rational functions, and thereby illustrate the general role played by the branch points
as being independent variables for analytic functions, besides the harmonic moments.

One way to handle the problem, mentioned in the beginning of Section~\ref{sec:non-univalent}, that the harmonic moments represent
too few test functions because functions in the $z$-plane cannot not distinguish points on different sheets on the Riemann surface above it,
is to turn to the class of quadrature Riemann surfaces. Such surfaces have been introduced and discussed in a special case in \cite{Sakai-1988} 
(see also \cite{Gustafsson-Tkachev-2011}), and we shall only need them in that special case. In principle, a quadrature Riemann surface is
a Riemann surface provided with a Riemannian metric such that a finite quadrature identity holds for the corresponding area integral 
of integrable analytic functions on the surface.  

The special case which we shall consider is that the Riemann surface is a bounded simply connected branched covering surface of the complex
plane. The Riemannian metric is then obtained by pull back from the complex plane via the covering map. In addition, we shall only
consider one point quadrature identities, but then of arbitrary high order. All this amounts to a generalization of polynomial images of the unit disk.

Being simply connected means that the quadrature Riemann surface is the conformal image of an analytic function $f$ in $\D$, and the one
point quadrature identity then is of the form, when pulled back to $\D$,
\begin{equation}\label{qiRS}
\frac{1}{2\pi\I} \int_\D g(\zeta) |f'(\zeta)|^2 \,d\bar{\zeta}d\zeta= \sum_{j=0}^n c_j g^{(j)}(0).
\end{equation}
This is to hold for analytic test functions $g$ which are integrable  with respect to the weight $|f'|^2$. The $c_j$ are fixed complex constants ($c_0$ necessarily real
and positive), and we assume that $c_n\ne 0$ to give the integer $n$ a definite meaning.

Such an identity (\ref{qiRS}) holds whenever $f$ is a polynomial, and  it is well known that in case $f$ is univalent, being a polynomial is actually necessary
for an identity (\ref{qiRS}) to hold. However, for non-univalent functions $f$ it is different. As we have already seen in Subsection~\ref{sec:second case}, rational functions
which are not polynomials can also give an identity (\ref{qiRS}) under certain conditions. 

Indeed, if (\ref{qiRS}) holds then $f$ has to be a rational function.
This has been proved in \cite{Sakai-1988}, and under our assumptions  it is easy to give a direct argument: applying (\ref{qiRS}) with
$$
g(\zeta)=\frac{1}{z-\zeta} 
$$
for $|z|>1$ makes the right member become an explicit rational function (see the right member of (\ref{qiRS1}) below), while the left member equals the Cauchy transform of the density $|f'|^2\chi_\D$. This
Cauchy transform is a continuous function in all of $\C$ with the $\bar{z}$-derivative equal to $\overline{f'(z)}f'(z)$ in $\D$, thus being there of the form
$\overline{f(z)}f'(z)+h(z)$, with $h(z)$ holomorphic in $\D$. The continuity then gives the matching condition
\begin{equation}\label{qiRS1}
\overline{f(z)}f'(z)+h(z)=\sum_{k=0}^n\frac{k! c_k}{z^{k+1}}, \quad z\in\partial\D.
\end{equation}
Here $\overline{f(z)}$ can be replaced by $f^*(z)$, and it follows that this function extends to be meromorphic in $\D$. Hence $f^*$ (and $f$) are meromorphic
on the entire Riemann sphere, and so rational.  

We point  out that  quadrature Riemann surfaces as above are dense in the class of all bounded simply connected branched covering surfaces over $\C$.
This is obvious since each polynomial $f$ produces such a surface. Therefore the restriction to quadrature Riemann surfaces is no severe restriction.
The reason that (\ref{qiRS}) is a useful identity in our context is that it reduces
the information of $f(\D)$ as a multi-sheeted surface to information concentrated at one single point on it, and near that point it does not 
matter that the test functions $1,z,z^2, \dots$ cannot distinguish different sheets from each other.    

Having an identity (\ref{qiRS}) we can easily compare the constants $(c_0,c_1, \dots, c_n)$ with the moments $(M_0,M_1,\dots,M_n)$. It is just to choose 
$g(\zeta)=f(\zeta)^k$ to obtain $M_k$, and this gives a linear relationship mediated by a non-singular triangular matrix. Thus we have a one-to one correspondence
\begin{equation}\label{cM}
 (M_0,M_1,\dots,M_n) \leftrightarrow(c_0,c_1, \dots, c_n).
\end{equation}
The relations between the $c_j$ and $f$ are obtained by reading off,
from (\ref{qiRS1}) with $\bar{f}$ replaced by $f^*$,  the Laurent expansion of $f^*f'$ at the origin:
\begin{equation}\label{fstarfprime}
f^*(\zeta)f'(\zeta)= \sum_{k=0}^{n} \frac{k!c_k}{\zeta^{k+1}} +\text{ holomorphic in }\D.
\end{equation}
This is to be combined with (\ref{cM}). We see that the information about moments are now encoded in local information of $f$ at the origin and infinity.

If $f'$ has zeros $\omega_1,\dots, \omega_m$ in $\D$ then (\ref{fstarfprime}) means that $f^*$ is allowed to have poles at these points, in 
addition to the necessary pole of order $n$ at the origin, which is implicit in (\ref{fstarfprime}) since $f'(0)\ne 0$. We may now start counting parameters. Taking into account the normalization
at the origin, $f$ has from start $(1+2m+2n)+2m$ real parameters (numerator plus denominator when writing $f$ as a quotient).   These shall be matched with the $1+2n$
parameters in the $M_k$ or $c_k$. Next, each pole of $f^*$ in $\D\setminus \{0\}$ has to be a zero of $f'$, which give $2m$ equations for the parameters.
Now there remain $2m$ free parameters, and we claim that these can be taken to be the locations of the branch points, namely
\begin{equation}\label{B}
B_j=f(\omega_j) \quad j=1,\dots, m.
\end{equation}
Thus we expect that $f$ can be parametrized by the $M_k$ and the $B_j$:
\begin{equation}\label{fMB}
f(\zeta)=f(\zeta; M_0, \dots, M_n; B_1,\dots, B_m).
\end{equation}
In particular, $\partial f/\partial M_0$ then makes sense, with the understanding that $B_1,\dots, B_m$, as well as $M_1,\dots,M_n$, 
are kept fixed under the derivation.

Clearly the parameters $M_k$ and $B_j$ depend smoothly on $f$. This dependence can be made explicit by obvious residue formulas:
\begin{align}\label{MB}
M_k&=\frac{1}{2\pi \I} \int_{\partial\D} f(\zeta)^k f^*(\zeta) f'(\zeta) d\zeta= \res_{\zeta=0} f(\zeta)^k f^*(\zeta) f'(\zeta) d\zeta,\\
B_j&=\frac{1}{2\pi \I} \oint_{|\zeta-\omega_j|=\varepsilon} \frac{f(\zeta) f''(\zeta)}{f'(\zeta)} d\zeta= \res_{\zeta=\omega_j}  \frac{f(\zeta) f''(\zeta)}{f'(\zeta)} d\zeta.
\label{MB2}\end{align}
Thus the branch points $B_1,\dots, B_m$ are exactly the residues of $ \frac{f(\zeta) f''(\zeta)}{f'(\zeta)} d\zeta$ in $\D$.

\begin{theorem}\label{thm:QRS} 
Consider functions $f$ which are analytic in a neighborhood of the closed unit disk, are normalized by $f(0)=0$, $f'(0)>0$ and satisfy $f'\ne 0$ on $\partial\D$. 
Let $\omega_1,\dots,\omega_m$ denote the zeros of $f'$ in $\D$, these zeros assumed to be simple.

Under these assumptions a  quadrature identity  of the kind (\ref{qiRS}) holds, for some choice of coefficients $c_0, c_1,\dots, c_n$ with $c_n\ne 0$,
if and only if $f$ is a rational function such that $f$ has a pole 
of order $n+1$ at infinity and possibly finite poles at the reflected points $1/\bar{\omega}_k$ of the zeros of $f'$ in $\D$. 
This means that $f$ is of the form
\begin{equation}\label{faomega}
f(\zeta)=\frac{a_0\zeta+a_1 \zeta^2+\dots+ a_{m+n}\zeta^{m+n+1}}{(1-\bar{\omega}_1\zeta)\dots (1-\bar{\omega}_m\zeta)}.
\end{equation}
Here $a_0,\dots, a_{m+n+1}$ may be viewed as free parameters (local coordinates for $f$), subject only to $a_0>0$, 
with the roots $\omega_1, \dots, \omega_m$ then being determined by  the conditions $f'(\omega_k)=0$.

Another set of local coordinates for the same space of functions are the harmonic moments $M_0,M_1,\dots, M_n$  together with the branch points $B_1, \dots,B_m$,
defined by (\ref{B}) or (\ref{MB2}).  Thus we can write $f$ (locally) on the form (\ref{fMB}),
with the particular consequence that the partial derivative $\partial f/\partial M_0$ and the Poisson bracket (\ref{poisson}) make sense. Finally,  the string equation
$$
\{f,f^*\}=1
$$
holds.
\end{theorem}

\begin{proof}
The first statements, concerning the form of $f$, were proved already in the text preceding the theorem. 

As remarked before the proof, the parameters $M_0,M_1,\dots,M_n; B_1,\dots, B_m$ represent as many data as there are independent coefficients
in (\ref{faomega}), namely the $a_0,a_1, \dots,a_{m+n}$, and we take for granted that they are indeed independent coordinates. 
From that point on there is a straight-forward argument proving the theorem based on existence result for the Hele-Shaw flow moving boundary problem
(or Laplacian growth).

Indeed, it is known that there exists, given $f(\cdot,0)$, an evolution $t\mapsto f(\cdot,t)$ such that $M_1, \dots, M_n; B_1, \dots, B_m$ remain fixed under the evolution, and  
\begin{equation}\label{stringt}
\{f,f^*\}_t=1,
\end{equation}
holds. The branch points being kept fixed means that the evolution is to take place on a fixed Riemann surface (which however has to be  extended during the evolution),
and then at least a weak solution forward in time ($t \geq 0$) can be guaranteed by potential theoretic methods (partial balayage or obstacle problems). 
Under the present assumptions, involving only rational functions, local solutions in both time directions can also be obtained by direct approaches
which reduce (\ref{stringt}) to finite dimensional dynamical systems, similar to those discussed in the polynomial case, Section~\ref{sec:polynomials}. 
Both these methods are developed in detail in \cite{Gustafsson-Lin-2014}.  Below we give some further details related to the direct  approach with rational functions.

From (\ref{stringt}) alone follows that the moments $M_1, \dots, M_n$ are preserved. This can be seen from
Corollary~\ref{lem:momentvariation}, which gives
\begin{equation}\label{dMdt}
\frac{d}{dt}M_k=\frac{1}{2\pi\I}\frac{d}{dt} \int_\D f(\zeta,t)^k|f(\zeta,t)|^2 d\bar{\zeta}d\zeta =
\end{equation}
$$
=\frac{1}{2\pi}\int_{0}^{2\pi}  f(\zeta,t)^k \{f,f^*\}_t \,d\theta
=\frac{1}{2\pi}\int_0^{2\pi} f^k d\theta =0
$$
for $k\geq 1$.
The equation (\ref{stringt}) is equivalent to 
$2{\rm Re\,} [\dot{f}(\zeta, t)\,\overline{\zeta f'(\zeta,t)} ]=1$ holding for  $\zeta\in\partial\D$, essentially the Polubarinova-Galin equation (\ref{PG}), 
and on dividing by $|f'(\zeta,t)|^2$ this becomes
\begin{equation}\label{PG1}
{\rm Re\,} \frac{\dot{f}(\zeta,t)}{\zeta f'(\zeta,t)} =\frac{1}{2|f'(\zeta,t)|^2}, \quad \zeta\in\partial\D.
\end{equation}
If  $f'$ had no zeros  in $\D$, then (\ref{PG1}) would give that
\begin{equation}\label{lk}
\dot{f}(\zeta,t)=\zeta f'(\zeta,t)P(\zeta,t)
\quad(\zeta\in\mathbb{D}),
\end{equation}
where $P(\zeta,t)$ is the Poisson-Schwarz integral
\begin{equation}\label{poisson1}
P(\zeta,t) = \frac{1}{2\pi }\int_{0}^{2\pi}\frac{1}{2|f'(e^{i\theta},t)|^2}\,\frac{e^{i\theta}+\zeta}{e^{i\theta}-\zeta}\,d\theta.
\end{equation}

With zeros of $f'$ allowed in $\D$, one may add to $P(\zeta,t)$
rational functions which are purely imaginary on $\partial\D$ and whose poles in $\D$
are killed by the zeros of the factor $\zeta f'(\zeta,t)$ in front. The result is 
\begin{equation}\label{lkgen}
\dot{f}(\zeta,t)=\zeta f'(\zeta,t)\left(P(\zeta,t)+R(\zeta,t)\right),
\end{equation}
where  $R(\zeta,t)$ is any function of the form
\begin{equation}\label{Rsum}
R(\zeta,t)=\I\im\sum_{j=1}^m\frac{b_j(t)}{\omega_j(t)}
+\sum_{j=1}^m\left(  \frac{b_{j}(t)}{\zeta -\omega_j(t)}-
\frac{\bar{b}_{j}(t)\zeta}{1 -\bar{\omega}_j(t)\zeta}\right).
\end{equation}
The first term here is just to ensure normalization, namely $\im R(0,t)=0$.
Thus (\ref{lkgen}), together with (\ref{poisson1}) and (\ref{Rsum}), is equivalent to (\ref{stringt}). 

The coefficients $b_j\in\C$ in (\ref{Rsum}) are arbitrary, and they are actually proportional to the speed of the branch points under variation of $t$:
using (\ref{lkgen}) and (\ref{Rsum}) we see that
$$
\frac{d}{dt}B_j= f'(\omega_j,t)\dot{\omega}_j+ \dot{f}(\omega_j,t)=\dot{f}(\omega_j,t)
=\omega_j b_{j} f''(\omega_j,t).
$$
Here $f''(\omega_j)\ne 0$ since we assumed that the zeros are simple.

Thus the term $R(\zeta,t)$ represents motions of the branch points. However, we are interested in the 
case that the branch points do not move, since the interpretation of $\partial/\partial M_0$ in the string equation
amounts to all other variables $M_1, \dots,M_n; B_1,\dots, B_m$  being kept fixed. Thus we have only the equation
(\ref{lk}) to deal with, and as mentioned at least local existence of solutions of this can be guaranteed by potential theoretic or 
complex analytic methods, see \cite{Gustafsson-Lin-2014}.

\end{proof}



\bibliography{bibliography_gbjorn.bib}

\end{document}